\documentclass{article}
\usepackage[utf8]{inputenc}

\usepackage{amsmath}
\usepackage{amsthm}
\usepackage{amssymb}
\usepackage{upgreek}
\usepackage{xcolor}
\usepackage{comment}
\usepackage{tikz-cd}
\usepackage{bbm} 
\usepackage{todonotes}

\theoremstyle{plain}
\newtheorem{thm}{Theorem}[section]
\newtheorem{lemma}[thm]{Lemma}
\newtheorem{prop}[thm]{Proposition}
\newtheorem{cor}[thm]{Corollary}
\theoremstyle{definition}
\newtheorem{defn}[thm]{Definition}
\newtheorem{example}[thm]{Example}
\newtheorem{qu}[thm]{Question}
\newtheorem{notation}[thm]{Notation}
\newtheorem{remark}[thm]{Remark}
\newtheorem{meth}{Method}

\def\C{\mathbb{C}}

\def\N{\mathbb{N}}

\def\ov{\overline}

\def\im{\text{im}}

\newcommand{\mat}{M}
\newcommand{\1}{\mathbbm{1}}

\title{A construction of deformations to general algebras}

\author{Dave Bowman, Dora Pulji\'c, Agata Smoktunowicz }

\begin{document}

\maketitle

\color{red}


\color{black}

\begin{abstract} One of the questions investigated in deformation theory is to determine to which algebras can a given associative algebra be deformed. In this paper we investigate a different but related question, namely: for a given associative finite-dimensional $\C$-algebra $A$, find algebras $N$ which can be deformed to $A$. We develop a simple method which produces associative and flat deformations to investigate this question. As an application of this method we answer a question of Michael Wemyss about deformations of contraction algebras.
\end{abstract}

\section*{Introduction}


Michael Wemyss and Will Donovan developed a method for characterising commutative rings which appear in algebraic geometry, by using methods from noncommutative ring theory. In \cite{donovan_noncommutative_2016} they introduced contraction algebras which provide important insight into resolutions of noncommutative singularities and invariants of flops. Contraction algebras can be described in a purely algebraic way, by using generators and relations. They appear in many questions which can be investigated by noncommutative ring theorists who are not familiar with advanced methods of algebraic geometry. Some of these questions are related to nilpotent rings. In 2022, Gavin Brown and Michael Wemyss described all finite dimensional contraction algebras which are local \cite{brown2022local}. Notice that the Jacobson radical of a finite dimensional algebra is nilpotent, and hence contraction algebras which are local are very close to being nilpotent. Some other questions involve characterisation of contraction algebras which are not local.

Deformations of contraction algebras give insight into invariants of flops. In this context Michael Wemyss asked questions about deformations of contraction algebras which are local. Using geometric methods he was able to conjecture which rings can be obtained as deformations of these contraction algebras.

Since contraction algebras are noncommutative, the usual methods using derivations as in \cite{coll_explicit_nodate} cannot be applied. Another approach can be to use Hochschild cohomology and the Gerstenhaber bracket, however this often leads to complicated calculations as seen in \cite{shepler_pbw_2013}.

In this paper, we develop a method to calculate deformations of noncommutative
local algebras. In particular we answer a question of Michael Wemyss regarding deformations of contraction algebras. For information about contraction algebras we refer reader to \cite{wemyss2023lockdown,donovan_noncommutative_2016}.

In \cite[Conjecture 4.3]{MR3805201} Hua and Toda conjectured that there exists an algebraic flat deformation of the contraction algebra. As an application we confirm this conjecture for one of the contraction algebras constructed in \cite{brown2022local}. Note that our method can be applied to other contraction algebras described in \cite{brown2022local}. Observe that Theorem $4.2$ of \cite{MR3805201} gives important information about deformations of contraction algebras which were obtained by using geometric methods. More information related to the existence of geometric deformations is given on pages 7-9 in \cite{MR3414491}.

In \cite{shepler_pbw_2013} deformations of graded Hecke algebras were constructed using novel methods. It is known that deformations of graded 
Hecke algebras give important information about the Hecke algebras.
Our method could be also used as another method of choice for constructing deformations of graded Hecke algebras. Some other methods of constructing deformations of associative algebras  were described in \cite{MR3971234}. 

\section{Preliminaries}

\begin{defn}
     Let $S$ be a subset of a $k$-algebra $A$ for a commutative unital ring $k$. Then $S$ is a \textbf{generating set} if all elements of $A$ can be written as a $k$-linear sum of products of elements of $S$ using the operations in $A$.
\end{defn}

\begin{defn}[\cite{MR3971234}]\label{defdef}
A \textbf{formal deformation} $(A_t,\ast)$ of a $k$-algebra $A$ is an associative $k$-bilinear multiplication $\ast$ on the $k[t]$-module $A[t]$, such that in the specification at $t=0$, the multiplication corresponds to that on $A$; this multiplication is required to be determined by a multiplication on elements of $A$ and extended to $A[t]$ by
the Cauchy product rule.
\end{defn}

Any multiplication $\ast$ as in Definition \ref{defdef} is determined by products of pairs of elements of $A$: for $a, b \in A$, write
\[a\ast b = \sum_{i\geq 0}\mu_i(a\otimes b)t^i,\]
where $\mu_0$ denotes the usual product in $A$ and $\mu_i$ for $i\geq 1$ are $k$-linear functions from $A\otimes A$ to $ A$.\\

\begin{defn}[\cite{feigin_flat_2019}]
    Let $A$ be a $k$-algebra and for $t\in \C$ let $\{A_t\}$ be the family of algebras arising from a deformation $*$ of $A$. Then $*$ is a \textbf{flat} deformation if each $A_t$ has the same dimension. 
\end{defn}

\begin{notation}
    We will denote by $\C[x,y]$ the polynomial ring over $\C$ in two non-commuting variables, $x$ and $y$.
\end{notation}

\begin{notation}
    Consider monomials of the ring $\C[x,y]$. We can order the monomials using shortlex ordering by specification $1<x<y$. We denote the monomials of $\C[x,y]$ by $p_1, p_1,p_2\dots$
where $p_i<p_{j}$ for $i<j$. 
\end{notation}

\begin{notation}
    We will denote by $\C[x,y][t]$ the polynomial ring in variable $t$ with coefficients from the polynomial ring in two non-commuting variables $x$ and $y$, with coefficients from $\C$.
\end{notation}

\begin{notation}
    For $a_i, b_i \in A$ we define $f:\C[x,y][t]\rightarrow A[t]$  to be a homomorphism of $\C$-algebras such that 
    \[f:x\mapsto \sum_{i=1}^na_it^i,\quad  f:y\mapsto \sum_{i=1}^nb_it^i, \quad f:t\mapsto t.\]
    We will denote $\ker f$ by $I$.
\end{notation}

\begin{notation}\label{not:pi}
 For an element $p_{i}$ in $\C[x,y]$ we will denote by $\ov{p_{i}}$ the same element $p_i$, but with all instances of $x$ replaced by $\sum_{i=1}^na_it^i$, and all instances of $y$ replaced by $\sum_{i=1}^nb_it^i$. We will denote by $\ov{\ov{p_{i}}}$ the same element as $p_i$, but with all instances of $x$ replaced by $\sum_{i=1}^na_i$, and all instances of $y$ replaced by $\sum_{i=1}^nb_i$. 
\end{notation}

\begin{notation}
   Let $R$ be a ring and $I$ be an ideal of $R$, then elements of the factor ring $R/I$ will be denoted as $r+I$, where $r\in R$.
 Notice that $r+I=s+I$ for $r,s\in R$ if and only if $r-s\in I$.
\end{notation}

\section{Algorithm}

In this section we shall describe the simplest case of our proposed method. This requires the data of a $\C$-algebra $A$ and a two element generating set $\{a, b\}$. We shall state the steps of this algorithm and then provide notes regarding important components. Then some examples will be displayed.

\begin{meth} \label{meth}
We fix a finite dimensional $\C$-algebra $A$ with two generators $a, b$. Then: 
\begin{enumerate}
    \item We consider $A[t]$, the polynomial algebra over $A$. We will use $t$ as the parameter of our deformation. We define: 
    \begin{align*}
        x &:= ta, \\ 
        y &:= tb.
    \end{align*}
    \item We calculate: 
    \begin{align*}
        x^2 &= t^2 a^2, \\ 
        xy &= t^2 ab, \\ 
        yx &= t^2 ba, \\ 
        y^2 &= t^2 b^2, \\ 
    \end{align*}
    and continue to calculate larger products of $x, y$. We shall then proceed with a Diamond Lemma\footnote{See section 1 of \cite{bergman_diamond_1978} } like decomposition of large products of $a, b$ in terms of smaller products of $a, b$. This will cause all elements of sufficiently large length to have a power of $t$ as a factor. In doing so we will obtain relations on $x, y$ and products thereof. We terminate this process when we have enough relations to decompose any large product into only multiples of $x,y$ and powers of $t$. Our relations will be given by polynomials $p_1, \dots, p_m \in \C[x,y][t]$ for some $m\in \N.$ This terminates finitely because $A$ finite dimensional.
    \item We present the algebra: 
    \begin{equation*}
        \mathcal{N} := {\C[x,y][t]}/{\left< p_1, \dots, p_m \right>}.
    \end{equation*}
    We note that sufficiently large products of $x, y$ will obtain a factor of $t$. Now we evaluate $t$ at various values in $T = [0, \infty)$. We denote by $N_s$ the algebra that arises from $\mathcal{N}$ by evaluating $t$ at $s$ and in particular we write $N = N_0$. By step (2) $N$ is local, in Section 5 we consider a more general method where $N$ is not necessarily local. The algebra $\mathcal{N}$ and the family of algebras $\{N_t\}$ are the output of this method. 
\end{enumerate}
\end{meth}

By Theorem \ref{thm formal deformation} the family $\{N_t\}$ is a deformation of $N$ and by Proposition \ref{iso} it holds that $A \in \{N_t\}$.

Changing the chosen generators $a, b$ can result in different algebras $N_0$ that have $A$ as an associative deformation.

\subsection{Notes on the method}
We give some notes before the examples: 
\begin{itemize}
    \item In step (1) we are trying to capture the behaviour of the elements in $A$, but in a way that is controlled by $t$. This enables us to make use of a presentation while still having access to our deformation parameter $t$. 
    \item We note that since $a, b$ generate $A$ some linear combination of products of $a, b$ will be the identity. In step (2) it is important we find a relation on $x, y, t$ and $1 \in A$. 
    \item In step (3) we discard $A$ entirely. However, because of the relations $p_i$, $x$ and $y$ ``remember" that they came from $A$ and that they form a generating set. This step results in something subtle: the monomials $x, y \in N$ are no longer dependent on $t$ (and thus do not vanish when we set $t = 0$) but their products maintain their dependence on $t$. 
    We have loosened our grip on $x, y$ just enough for the deformation to take place.
    \item We note that our deformation will be associative since the multiplication of $N_t$ will inherit from the associative multiplication $*$ of the algebra $\mathcal{N}$. 
    \item This algorithm is easily generalised to algebras that require a larger generating set. Suppose a finite dimensional $\C$-algebra $A$ has $n$ generators $\{a_1, \dots, a_n\}$. Then in step (1) we would define $x_i := t a_i$ for $1 \leq i \leq n$ and proceed as described in Method \ref{meth}. The resulting family of algebras will also be a flat and associative deformation of $A$. 
\end{itemize}

\subsection{Examples}

Here we shall work through some examples. We shall start with a simple example where $A = \C \oplus \C$ and then give a more complicated one where $A = \mat_2(\C)$, the algebra of $2 \times 2$ matrices over $\C$.  
Practically it can be useful to fix a basis for the algebra $A$ as this helps find the decompositions in step (2). 

\begin{example}
 We consider $A = \C \oplus \C$ and fix $a = (i, 0), \; b = (0, 1)$. Clearly $\C$-linear combinations of $a$ and $b$ span $A$ so our basis is just $\{a, b \}$. 
We write $\1 := (1, 1)$. Now we consider $A[t]$ and define: 
\begin{align*}
    &x := ta = (ti, 0), \\ 
    &y := tb = (0, t).
\end{align*}
We calculate:
\begin{align*}
    &x^2 = t^2 a^2 = (-t^2, 0) = tix, \\ 
    &xy = t^2ab = (0, 0) = 0, \\ 
    &yx = t^2ab = (0, 0) = 0, \\ 
    &y^2 = t^2b^2 = (0, t^2) = ty,\\ 
    &t \1 = -ix + y.
\end{align*}
Thus we have relations: 
\begin{align*}
    p_1 &= x^2 - tix = 0, \\ 
    p_2 &= xy = 0, \\ 
    p_3 &= yx = 0, \\ 
    p_4 &= y^2 - ty = 0, \\ 
    p_5 &= t \1 -(-ix + y).
\end{align*}
    
Since $xy = yx = 0$ and squares of $x, y$ can be reduced,  we have enough relations to decompose any large product into only multiples of $x, y$ and powers of $t$. 

Thus we present: 
\begin{equation*}
    N := {\C[x,y][t]}/{\left<p_1, p_2, p_3, p_4, p_5\right>}.
\end{equation*}
We observe: 
\begin{equation*}
    N_0 = {\C[x, y]}/{\left<x^2, xy, yx, y^2, -ix + y \right>},
\end{equation*}
 and 
\begin{equation*}
    N_1 = {\C[x, y]}/{\left<x^2 - ix, xy, yx, y^2 - y, 1 + ix -y  \right>.}
\end{equation*}
We see that $x, y$ behave exactly as $a, b \in A$ and by Proposition \ref{iso}. It is easily checked that the isomorphism is given by: 
\begin{align*}
    \phi: N_1 &\to A \\
          1 &\mapsto \1 \\ 
          x &\mapsto (i, 0) \\ 
          y &\mapsto (0, 1).
\end{align*}
Thus we have produced an algebra $N_0$ that has $A = \C \oplus \C$ as a deformation. 
\end{example}

\begin{example}
We consider $A = \mat_2(\C)$ and fix: 
\begin{equation*}
    a = \begin{pmatrix}
           1 & 0 \\
           0 & -1
         \end{pmatrix}, \quad b = \begin{pmatrix}
           0 & 1 \\
           1 & 0
         \end{pmatrix}.
\end{equation*}
We denote by $\1$ the multiplicative unit in $A$. Next we compute:
\begin{align*}
    &\begin{pmatrix}
           1 & 0 \\
           0 & 0
         \end{pmatrix} = \frac{1}{2}(a + b^2) = e_1, \\ 
    &\begin{pmatrix}
           0 & 1 \\
           0 & 0
         \end{pmatrix} = \frac{1}{2}(ab + b) = e_2, \\ 
    &\begin{pmatrix}
           0 & 0 \\
           1 & 0
         \end{pmatrix} = \frac{1}{2}(b - ab) = e_3, \\ 
    &\begin{pmatrix}
           0 & 0 \\
           0 & 1
         \end{pmatrix} = \frac{1}{2}(b^2 - a) = e_4
\end{align*}
and thus $a, b$ generate $A$ as an algebra.
Now we consider $A[t]$ and define:
\begin{align*}
    &x := ta, \\ 
    &y := tb.
\end{align*}
We compute:
\begin{align*}
    &x^2 = t^2 \begin{pmatrix}
           1 & 0 \\
           0 & 1 
         \end{pmatrix} = t^2 \1, \\ 
    &xy = t^2 \begin{pmatrix}
           0 & 1 \\
           -1 & 0
         \end{pmatrix} = t^2(e_2 - e_3), \\ 
    &yx = t^2 \begin{pmatrix}
           0 & -1 \\
           1 & 0
         \end{pmatrix} = t^2(e_3 - e_2), \\ 
    &y^2 = t^2 \begin{pmatrix}
           1 & 0 \\
           0 & 1 
         \end{pmatrix} = t^2 \1 .
\end{align*}
We have obtained the relations: 
\begin{align*}
    &p_1 = x^2 - t^2 \1, \\ 
    &p_2 = x^2 - y^2 = 0, \\ 
    &p_3 = xy + yx = 0, \\ 
    &p_4 = y^2 - t^2 \1.
\end{align*}
Now we consider larger products to produce more relations: 
\begin{align*}
    &x^3 = t^3 \begin{pmatrix}
           1 & 0 \\
           0 & -1
         \end{pmatrix} = t^2x, \\
    &y^3 = t^3 \begin{pmatrix}
           0 & 1 \\
           1 & 0
         \end{pmatrix} = t^2y.
\end{align*}
Hence we have also obtained: 
\begin{align*}
    &p_5 = x^3 - t^2x = 0, \\ 
    &p_6 = y^3 - t^2y = 0.
\end{align*}
These allow us to derive the following: 
\begin{align*}
    &x^2y = y^3 = yx^2 = y^3 = t^2y, \\ 
    &y^2x = x^3 = xy^2 = x^3 = t^2x,
\end{align*}
which are enough to reduce an arbitrary product of $x, y$. Thus we present: 
\begin{equation*}
    N := {\C[x,y][t]}/{\left<p_1, p_2, p_3, p_4, p_5, p_6 \right>}.
\end{equation*}

We observe: \[N_0 = {\C[x, y]}/{\left< x^2, y^2, xy + yx \right>},\] 
and \[N_1 = {\C[x, y]}/{\left< x^2 -1, y^2 -1, xy + yx, x^3 -x, y^3 -y \right>}.\]
By Proposition \ref{iso} we have $N_1 \cong A$. It is easily checked that the isomorphism is given by: 
\begin{align*}
    \phi: N_1 &\to A \\
          1 &\mapsto \1 \\ 
          x &\mapsto \begin{pmatrix}
           1 & 0 \\
           0 & -1
         \end{pmatrix} \\ 
          y &\mapsto \begin{pmatrix}
           0 & 1 \\
           1 & 0
         \end{pmatrix}.
\end{align*}
Thus we have produced an algebra $N_0$ that has $A = \mat_2(\C)$ as a deformation. 
\end{example} 

\section{Deformations to $A$}\label{sec:method}
We now consider the general case in which $A$ is an associative finite dimensional $\C$-algebra generated by some elements $\sum_{i=1}^n a_i, \sum_{i=1}^n b_i$ for some $a_i,b_i\in A$.


\subsection{Specification at $t=0$}

In this section we define a multiplication on $\C[x,y][t]$ which gives the algebra $N$ at specification $t=0$.


\begin{prop}\label{basis}
    Let $\C[t]$ denote the polynomial ring in variable $t$.  Further, assume that for every $r\in A$ there exists $i=i(r)$ such that $rt^i\in \im f$. There exist elements $q_{1}, q_{2}, \dots , q_{n}\in {\mathbb C}[x,y][t]$ such that for every $k$ we have
    
    \begin{equation}\label{eq:basis}
        f(p_k)\in \sum_{i=1}^{n}\C[t] f(q_{i}).
    \end{equation}
    Moreover $n=\dim A.$ 
    
\end{prop}

\begin{proof}
We let $\{e_1,\dots, e_n\}$ be a basis of $A$. Let $k_1$ be the smallest possible integer such that for some integer $j_1$ we have
\[e_{j_1}t^{k_1}+\sum_{\substack{i=1 \\ i\neq j_1}}^n\alpha_i(t)e_i \in \im f\]

for some $\alpha_i(t)\in \C[t]$. Let $k_2$ be the smallest possible integer such that for some integer $j_2\neq j_1$ we have
\[e_{j_2}t^{k_{2}}+\sum_{\substack{i=1 \\ i\neq j_1, j_2}}^n\alpha'_i(t)e_i \in \im f\]
for some $\alpha'_i(t)\in \C[t]$. Continuing in this way, we define a basis 

$\{q_{1},\dots, q_{n}\}$ such that
\[
    f(q_{i}) = e_{j_m}t^{k_{m}}+\sum_{\substack{i=1 \\ i\neq j_1, \dots, j_m}}^n\alpha_i(t)e_i.
\]

Now notice that $f(q_1), f(q_2),\dots, f(q_n)$ are linearly independent over $\C$ as $e_1, e_2, \dots, e_n$ are linearly independent over $\C$.

We now prove that $f(q_{1})+t\im f,\dots, f(q_{n})+t\im f $ span the $\C$-algebra $\frac{\im f}{t\im f}$ over $\C$, and hence they are a basis for $\frac{\im f}{t\im f}$. 
Let $z_1\in \im f$ and write 
\[z_1=\sum_{i=1}^n r_i(t) e_{j_i}t^{s_i}\]
for some $r_i\in \C[t] - t\C[t]$ and some $j_i,s_i\in \N$. 

Notice that for the smallest $i$ such that $r_i(t)\neq 0$, $s_i\geq k_i$ for $k_i$ as above. Let
\[z_2 = z_1 - f(q_{1})t^{s_1-k_1}r_1(t).\]
 Note that we have
\[z_2\in \sum_{i=2}^n \C[t]e_{j_i}\]
so we can write
\[z_2 = \sum_{i=2}^n r^2_i(t) e_{j_i}t^{s^2_i}\]
for some $r^2_i(t)\in \C[t]-t\C[t]$ and $s^2_i\in \N$. Notice that $q_2\in \im f$. Now let 
\[z_3 = z_2 - f(q_{2})t^{s^2_2-k_2}r^2_2(t).\]

Continuing in this way we eventually arrive at $q_{n+1}=0.$ Hence by

 summing the following equations 
\begin{align*}
    0 &= z_n - f(q_{n})t^{s^n_n-k_n}r^n_n(t),\\
    & \vdots \\
    z_3 &= z_2 - f(q_{2})t^{s^2_2-k_2}r^2_2(t)
\end{align*}
 and 
\[z_2 = z_{1}-f(q_{1})t^{s_1-k_1}r_1(t)\]
we arrive at 
\[z_1=\sum_{l=1}^n f(q_{l})t^{s'_l-k_l}r^{l}_l(t)\in \sum_{l=1}^n\C[t]f(q_{l})\]
for some $s'_l\in\N$ as required.
\end{proof}

\textbf{Remark.} Notice that for $k\in \N$ and $\alpha_j\in \C[t]$ \[f(p_k)-\sum_{j=1}^n\alpha_j f(q{_j})=0\]
if and only if
\[p_k -\sum_{j=1}^n\alpha_j q_{j}\in I.
\]

\begin{prop}\label{prop:rel}
    For any $k, m\in\{1,2,\dots, n\}$ there exist $\zeta _{i,k}\in \mathbb C$, $\xi _{i,k}(t)\in {\mathbb C}[t]$ such that 
    \[p_{k}-\sum_{i=1}^{n}(\zeta _{i,k}q_{i}+t\xi_{i,k}(t)q_{i})\in I.\]
    In particular, there exist $\zeta _{i,k,m}\in \mathbb C$, $\xi _{i,k,m}(t)\in {\mathbb C}[t]$ such that 
\[q_{k}q_{m}-\sum_{i=1}^{n}(\zeta _{i,k,m}q_{i}+t\xi_{i,k,m}(t)q_{i})\in I.\]
\end{prop}

\begin{proof}

    For each element $p_{k} \in {\mathbb C}[x,y][t]$ we have 
\[p_{k}\in \sum_{i=1}^{n} {\mathbb C}[t]q_{i}+I.\]
 This follows by taking $f^{-1}$ of equation (\ref{eq:basis}). Hence, for each $k\in \N$ there exist $\sigma _{i,k}\in \mathbb C [t]$ such that 
\[p_{k}-\sum_{i=1}^{n}\sigma_{i,k}q_{i}\in I,\]
 where $I=\ker f$. Notice that we can write  $\sigma _{i,k}(t)=\zeta _{i,k} +  t\xi_{i,k}(t)$ for some $\zeta _{i,k}\in \mathbb C$, $\xi _{i,k}(t)\in {\mathbb C}[t]$, as to separate the terms dependent on $t$. The above then becomes
\[p_{k}-\sum_{i=1}^{n}(\zeta _{i,k}q_{i}+t\xi_{i,k}(t)q_{i})\in I.\]

\end{proof}

Reasoning similarly as in the proof of Proposition \ref{basis} we get the following corollary.

\begin{cor} \label{1}
Suppose $\sum_{i=1} ^n \alpha _{i}q_{i}\in I$ for some $\alpha _{i}\in {\mathbb C}[t]$. Then $f(\sum_{i} ^n \alpha _{i}q_{i})=0$ and consequently  $\alpha _{i}=0$ for every $i\in\{1,2,\dots, n\}$.

 Moreover, all elements of $I$ are $\C[t]$-linear combinations
 of elements 
\[p_{k}-\sum_{i=1}^{n}(\zeta _{i,k}q_{i}+t\xi_{i,k}(t)q_{i})\]
for some $k\in \N, \zeta _{i,k}\in \mathbb C$, $\xi _{i,k}(t)\in {\mathbb C}[t]$.
\end{cor}

\begin{notation}\label{not:N}
We denote by $J$ be the set consisting of $\C$-linear combinations of elements 
\begin{equation}\label{eq:J}
    p_{k}-\sum_{i=1}^{n}\zeta _{i,k}q_{i}\in {\mathbb C}[x,y],
\end{equation}
where $\zeta _{i,k}$ are as in Proposition \ref{prop:rel}.
 Let $\langle J\rangle $ be the ideal of ${\mathbb C}[x,y]$ generated by elements from $J$. Further, let $N$ be the quotient algebra $ \C[x,y]/\langle J\rangle$. 
 \end{notation}

\begin{cor}\label{2}
We have $e\in J$ if and only if  $e+e'\in I$ for some $e'\in t{\mathbb C}[x,y][t]$. In particular $J\subseteq I+t{\mathbb C}[x,y][t]$, and hence $\langle J\rangle \subseteq I+t{\mathbb C}[x,y][t]$.  
\end{cor}

\begin{prop}\label{3} The dimension of $N$ equals $n$. Moreover,  elements $q_{k}+J\in N$ are a basis of $N$ as a $\C$-vector space.
\end{prop}

\begin{proof}
    
 Notice that expression (\ref{eq:J}) implies that 
 \[{\mathbb C}[x,y]\subseteq \sum_{k=1}^{n}\C q_{k}+\langle J\rangle,\] and hence  
$\{q_{k}+\langle J\rangle\mid k\in \{1,2, \ldots ,n\}\}$  span $N$ as a $\C$-vector space.
 Therefore,  the dimension of $N$ does not exceed $n$.

We now show that the elements $q_{k}+\langle J\rangle\in N$ for $k\in \{1,2, \ldots ,n\}$ are linearly independent over $\mathbb C$.
Suppose on the contrary that we have
\[\sum_{i=1}^{n}\xi_{i}q_{i}\in \langle J\rangle \]
for some $\xi _{i}\in \mathbb {C}$. By Corollary \ref{2} and Proposition \ref{basis} we have  \[\langle J\rangle \subseteq I+t{\mathbb C}[x,y][t]\subseteq I+t\sum_{i=1}^{n}{\mathbb C}[t]q_{i}.\]
 It follows that there is $e'\in t\sum_{i=1}^{n}{\mathbb C}[t]q_{i}$ such that $\sum_{i=1}^{n}\xi _{i}q_{i}-e'\in I$, contradicting Corollary \ref{1}. 
\end{proof}

 Observe now that the following holds:
\begin{prop}\label{7} We can present $N$ as an $\mathbb C$-algebra 
 with generators $d_{1},\dots, d_n$, which span $N$ as a $\C$-vector space, subject to relations 
\[ d_{k}d_{m}=\sum_{i=1}^{n}\zeta _{i,k,m}d_{i}\]
where $\zeta_{i,k,m}$ are as in Proposition \ref{prop:rel}.

\end{prop}
\begin{proof} 

We let ${\mathbb C}[q_{1}, \ldots , q_{n}]$ be the  subalgebra of ${\mathbb C}[x,y]$ generated by elements $q_{1}, \ldots , q_{n}\in {\mathbb C}[x,y]$. Similarly, let ${\mathbb C}[d_{1}, \ldots , d_{n}]$ be the free $\C$-algebra generated by elements $d_{1}, \ldots , d_{n}$.  Let $\zeta:{\mathbb C}[d_{1}, \ldots , d_{n}]\rightarrow {\mathbb C}[x,y]/\langle J\rangle $ be defined as $\zeta (d_{i})=q_{i}+\langle J\rangle$.

We now show that $x+\langle J\rangle,y+\langle J\rangle\in  \im\zeta$. Notice that since $f(x)\in \im f$ we have
\[f(x)\in \sum_{i=0}^{\infty}{\mathbb C}[t]f(q_i).\]
Hence we have $f(x)=f(x')$ for some $x'\in \sum_{i=0}^{\infty}{\mathbb C}[t]f(q_i). $ Then $x-x'\in I$, and so $x-x''\in J$ for some $x''\in \sum_{i=0}^{\infty}{\mathbb C}q_i.$  Then $x''+\langle J\rangle \in \im \zeta$, and so $x+\langle J\rangle \in \im \zeta$. An analogous argument shows $y+\langle J\rangle \in \im \zeta$.

 Let \[J'=\ker (h )\subseteq {\mathbb C}[d_{1}, \ldots , d_{n}].\]
 Then by the First Isomorphism Theorem for rings ${\mathbb C}[d_{1}, \ldots , d_{n}]/J'$ is isomorphic to 
$ {\mathbb C}[x,y]/\langle J\rangle =N$. 

Let $J''$ be the ideal of ${\mathbb C}[d_{1}, \ldots , d_{n}]$ generated by  elements 
\[ d_{k}d_{m}-\sum_{i=1}^{n}\zeta _{i,k,m}d_{i}\]
where $\zeta_{i,k,m}$ are as in Proposition \ref{prop:rel}. Observe that 
\[\zeta \left (d_{k}d_{m}-\sum_{i=1}^{n}\zeta _{i,k,m}d_{i} \right)=q_{k}q_{m}-\sum_{i=1}^{n}\zeta _{i,k,m}q_{i}+\langle J\rangle =0+\langle J\rangle .\] 
 Therefore $J''\subseteq J'$. It follows that  the dimension of $ {\mathbb C}[d_1,\dots,d_n]/J'$ does not exceed the dimension of $ {\mathbb C}[d_1,\dots, d_n]/J''$. 
 
 We now show that $J'=J''$. Notice that ${\mathbb C}[d_{1}, \ldots , d_{n}]/J''$ is at most $n$-dimensional, since every element can be presented as a linear combination of elements $d_{i}+J''$ for $i=1,2, \ldots, n$. On the other hand 
${\mathbb C}[d_{1}, \ldots , d_{n}]/J'$ is isomorphic to 
 ${\mathbb C}[x,y]/\langle J \rangle $, and hence is $n$-dimensional (by Proposition \ref{3}). Hence by comparing dimensions we get  that 
 the dimension of  $ {\mathbb C}[d_1,\dots, d_n]/J''$ does not exceed the dimension of $ {\mathbb C}[d_1,\dots, d_n]/J'$. Previously we showed that $ {\mathbb C}[d_1,\dots, d_n]/J'$ does not exceed the dimension of $ {\mathbb C}[d_1,\dots, d_n]/J''$. It follows that 
$J'=J''$, as required. 
\end{proof}

We now define a multiplication on $N$ which gives a formal deformation to the algebra $A$.

\begin{thm} \label{thm formal deformation}
Let $d_1,\dots,d_n$ be free generators of the $\C$-algebra $\C[d_1,\dots,d_n]$ and suppose $\zeta _{i,k,m}\in \mathbb C$, $\xi _{i,k,m}(t)\in {\mathbb C}[t]$ are as in Proposition \ref{prop:rel}. Then the multiplication rule
\[d_{k} \ast_t d_{m}=\sum_{i=1}^{n}(\zeta _{i,k,m}d_{i}+t\xi_{i,k,m}(t)d_{i}).\]
 gives a formal deformation such that $\ast_0$ gives the multiplication on an algebra isomorphic to $N$. 
\end{thm}

\begin{proof}
    
 Recall relations from Proposition \ref{prop:rel} that for any $k, m\in\{1,2,\dots, n\}$  there exist $\zeta _{i,k,m}\in \mathbb C$, $\xi _{i,k,m}(t)\in {\mathbb C}[t]$ such that 
\begin{equation}\label{eq:relI}
    q_{k}q_{m}-\sum_{i=1}^{n}(\zeta _{i,k,m}q_{i}+t\xi_{i,k,m}(t)q_{i})\in I.
\end{equation}
 
We introduce notation $[q_{i}]:=q_{i}+I$ for elements of ${\mathbb C}[x,y]/I$. Hence we obtain the following relations corresponding to (\ref{eq:relI}) which give the  multiplicative table  on ${\mathbb C}[x,y]/I$. We have
\[[q_{k}] [q_{m}]-\sum_{i=1}^{n}(\zeta _{i,k,m}[q_{i}]+t\xi_{i,k,m}(t)[q_{i}]).\]
 
Notice that these relations give a multiplication
\begin{equation}\label{eq:def0}
  d_{k} \ast_t d_{m}=\sum_{i=1}^{n}(\zeta _{i,k,m}d_{i}+t\xi_{i,k,m}(t)d_{i})  
\end{equation}
for $d_i\in \C[x,y].$ Now recall that $N$ is isomorphic to ${\mathbb C}[d_{1}, \ldots , d_{n}]/J''$ where $J''$ is the ideal of 
${\mathbb C}[d_{1}, \ldots , d_{n}]$ generated by relations 
\[d_{k}d_{m}-\sum_{i=1}^{n}\zeta _{i,k,m}d_{i}.\] 
Therefore by setting $t=0$ in (\ref{eq:def0}) we obtain the multiplication rule for ${\mathbb C}[d_{1}, \ldots , d_{n}]/J''$. 
\end{proof}

\subsection{Specification at $t=1$}
In this section we define a multiplication on $\C[x,y][t]$ which gives a formal deformation of $N$, such that at specification $t=1$ we get an algebra isomorphic to $A$.

\begin{prop}\label{iso}
    Let $\xi:\C[x,y][t]\rightarrow A$ be a homomorphism of $\C$-algebras such that
    \[\xi(x)=\sum_{i=0}^na_i, \quad \xi(y)=\sum_{i=0}^nb_i,\quad \xi(t)=1.\]
    Further, assume that for every $r\in A$ there exists $i=i(r)$ such that $rt^i\in \im f$. Then $\ker \xi=\langle I, t-1\rangle.$
\end{prop}

\begin{proof}
    We let $e=\sum_{i}\alpha_i p_{i}t^{\beta{i}}\in \ker \xi$ where $p_{i}$ are monomials in $\C[x,y]$, $\alpha_i\in \C$ and $\beta_i\in \N$. We will show there exists $\gamma_i\in \N$  such that $\hat{e}:=\sum_{i}\alpha_i p_{i}t^{\gamma_i}\in I$, so that $e\in I+\langle t-1\rangle$ as $e-\hat{e}\in \langle t-1\rangle.$ This follows since 
    \[e-\hat{e} = \sum_{i}\alpha_i p_{i}t^{\beta_i}-\sum_{i}\alpha_i p_{i}t^{\gamma_i} = \sum_{i}\alpha_i p_{i}(t^{\beta_i}-t^{\gamma_i}) = \sum_{i}\alpha_i p_{i} t^m(t^n-1)\in \langle t-1\rangle\]
    for some $m,n\in \N$.

    Recall Notation \ref{not:pi} and note that $\xi(e) = \sum_{i}\alpha_i \ov{\ov{p_{i}}} =0$. Notice that by assumption there exist $j_i\in\N$ such that $\ov{\ov{p_i}}t^{j_i}\in \im f$. We let $f(c_i) = \ov{\ov{p_i}}t^{j_i}$ for some $c_i\in\C[x,y][t]$. Hence for large enough $k\in \N$ we have \[f\left(\sum_{i}\alpha_i c_{i} t^{k-j_i}\right)=\sum_{i}\alpha_i \ov{\ov{p_{i}}}t^{k}=0.\]
\end{proof}

\begin{prop}\label{prop:10} Let notation be as in Proposition \ref{iso}, and suppose that $\sum_{i=0}^na_i$ and $\sum_{i=0}^nb_i$ generate $A$ as a $\mathbb C$-algebra.
    We have 
    \[\frac{{\mathbb C}[x,y][t]/I}{\langle t-1+I \rangle}\cong A.\]
\end{prop}

\begin{proof}
    Note that the ideal $\langle t-1+I \rangle$ of ${\mathbb C}[x,y][t]/I$ equals the set 
    \[\frac{{\mathbb C}[x,y][t](t-1)+I}{I} = \frac{\langle I, t-1\rangle}{I}  = \{g(t-1)+I\mid g\in {\mathbb C}[x,y][t]\}.\]
    
    By the Third Isomorphism Theorem we have
    \[\frac{{\mathbb C}[x,y][t]/I}{\langle I, t-1\rangle /I} \cong \frac{{\mathbb C}[x,y][t]}{ \langle I, t-1\rangle }.\]
    By Proposition \ref{iso} we have $\langle I, t-1\rangle =\ker\xi$ and as $\xi$ is onto, 
    \[\frac{{\mathbb C}[x,y][t]}{\ker\xi}\cong A\]
    by the First Isomorphism Theorem.
\end{proof}

We now define a multiplication on ${\mathbb C}[x,y][t]/\langle t-1,I\rangle  $ which gives a deformation to $A$ at $t=1$. 

\begin{thm}\label{thm:iso1}
Let notation be as in Proposition \ref{iso}, and suppose that $\sum_{i=0}^na_i$ and $\sum_{i=0}^nb_i$ generate $A$ as a $\mathbb C$-algebra. 
Let $d_1,\dots,d_n$ be free generators of the $\C$-algebra $\C[d_1,\dots,d_n]$ and suppose $\zeta _{i,k,m}\in \mathbb C$, $\xi _{i,k,m}(t)\in {\mathbb C}[t]$ are as in Proposition \ref{prop:rel}. Then the multiplication rule
\[d_{k} \ast_t d_{m}=\sum_{i=1}^{n}(\zeta _{i,k,m}d_{i}+t\xi_{i,k,m}(t)d_{i}).\]
 gives a formal deformation such that $\ast_1$ gives the multiplication on an algebra generated by $d_1,\dots,d_n$ isomorphic to ${\mathbb C}[x,y][t]/\langle I, t-1\rangle $. 
\end{thm}

\begin{proof}
We show that the algebra ${\mathbb C}[x,y][t]/\langle I, t-1\rangle $ is isomorphic to the algebra
 ${\mathbb C}[d_{1}, d_{2}, \ldots d_{n}]/I'$ where $I'$ is the ideal  of ${\mathbb C}[d_{1}, d_{2}, \ldots d_{n}]$ generated by elements 
\[d_{k} \ast_1 d_{m}-\sum_{i=1}^{n}(\zeta _{i,k,m}d_{i}+\xi_{i,k,m}(1)d_{i})\]
for $\zeta _{i,k,m}\in \mathbb C$, $\xi _{i,k,m}(t)\in {\mathbb C}[t]$  as in Proposition \ref{prop:rel}. 

We let 
$\delta: {\mathbb  C}[d_{1}, \ldots , d_{n}]\rightarrow {\mathbb C}[x,y][t]/\langle I, t-1\rangle $ be such that
\[\delta (d_{i})=q_{i}+\langle I, t-1\rangle ,\]
 for $i\in \{1,2, \ldots , n\}$. Observe that $I'\subseteq \ker(\delta )$ since
\begin{align*}
    \delta \left(d_{k} * d_{m}-\sum_{i=1}^{n}(\zeta _{i,k,m}d_{i}+\xi_{i,k,m}(1)d_{i}) \right)&= q_{k} * q_{m}-\sum_{i=1}^{n}(\zeta _{i,k,m}q_{i}+\xi_{i,k,m}(1)q_{i})\\
    &\quad +\langle I, t-1\rangle \\
    &=\langle I, t-1\rangle 
\end{align*}
 
 since $q_{k} \ast q_{m}-\sum_{i=1}^{n}(\zeta _{i,k,m}q_{i}+\xi_{i,k,m}(1)q_{i})\in I+\langle t-1\rangle $. 

Therefore the dimension of ${\mathbb  C}[d_{1}, \ldots , d_{n}]/\ker(\delta )$ does not exceed the dimension of  ${\mathbb  C}[d_{1}, \ldots , d_{n}]/I'$.
Notice that   ${\mathbb  C}[d_{1}, \ldots , d_{n}]/I' $ is spanned by elements $d_{i}+I$ as a vector space, and hence has dimension at most $n$.
On the other hand, by the First Isomorphism Theorem for rings: \begin{equation} \label{lines 5 prop 3.7}
    {\mathbb  C}[d_{1}, \ldots , d_{n}]/ker(\delta ) \cong \im(\delta)={\mathbb C}[x,y][t]/\langle I, t-1\rangle,
\end{equation} which in turn is isomorphic to $A$  by Proposition \ref{iso}, and hence has dimension $n$.
 It follows that $I'=\ker(\delta)$ and hence ${\mathbb  C}[d_{1}, \ldots , d_{n}]/ker(\delta )$ is isomorphic to ${\mathbb  C}[d_{1}, \ldots , d_{n}]/I'$. Where the equality in Equation \ref{lines 5 prop 3.7} is true by the proof of Proposition \ref{7}.
    
\end{proof}

\begin{remark}{\em Deformations at $t\neq 0$} The results from this section, hold with analogous proofs for a specialisation at other $t\neq 0$.
 Let $z \in \mathbb C$ and $z \neq 0$, the the 
 deformation at $t=z$ is also isomorphic to $A$, provided that
 $\sum_{i=0}^na_iz^{i}$ and $\sum_{i=0}^nb_iz^{i}$ generate $A$ as a $\mathbb C$-algebra. 
 \end{remark}

\begin{remark}
    Since the dimension of all the algebras in $\{N_t\}_{t \in [0, 1]}$ remains $n = \dim(A)$ the deformation arising from $*$ is flat. 
\end{remark}

\section{Application}

In this section we prove a conjecture of M. Wemyss using the method in section \ref{sec:method}. The statement is that of the following theorem.

\begin{thm}\label{thm:wem}
 Let 
 \[N = \frac{\C[x,y]}{\langle xy+yx, x^3+y^2, y^3\rangle}.\]
 Then $N$ deforms into  $A=M_2(\C)\oplus \C^{\oplus 5}$. 
\end{thm}

We prove this theorem through a series of lemmas. Firstly, we fix notation in addition to the notation of Theorem \ref{thm:wem}.

\begin{notation}
    Let $i_1,i_2$ be roots of the polynomial $x^2-x+1$.
\end{notation}

\begin{notation}\label{not:alpha}
    Let $e\in \C$ be such that for $c=\frac{i_1-i_2}{\sqrt{1+e^2}}$ the polynomial 
    \[g(z) = (2z-1)^2z^3 +c^2\]
    has $5$ distinct roots $\alpha_1, \dots, \alpha_5\in \C$. Further, we assume that $e\neq 0$, $c^2\neq 3$ and $\alpha_i\neq\frac{1}{2}.$ Notice that $g(z)$ and $z^3+1$ have no common roots, so $\alpha_i\neq -1$ and $\alpha_j\notin\{i_1,i_2\}$.
\end{notation}

\begin{notation}\label{not:alpha,beta}
    Let $\alpha = (\alpha_1,\alpha_2,\alpha_3, \alpha_4, \alpha_5)$ and $\beta=(\beta_1,\beta_2,\beta_3, \beta_4, \beta_5)$  for $\beta_i\in\C$ such that 
    \[(\beta_1,\beta_2,\beta_3, \beta_4, \beta_5)(2\alpha_1-1,2\alpha_2-1, 2\alpha_3-1, 2\alpha_4-1, 2\alpha_5-1 ) = c(1,1,1,1,1).\]
    Notice that such $\beta_i$ exist since $\alpha_i\neq \frac{1}{2}$.
\end{notation}

\begin{lemma}
Let $\xi: \C[x,y][t]\rightarrow A$ be a homomorphism of $\C$-algebras such that
\begin{align*}
    \xi(x) &= \begin{pmatrix}
           i_1 & 0 \\
           0 & i_2
         \end{pmatrix}
         +  \alpha, \quad
    \xi(y) = \begin{pmatrix}
           1 & e \\
           e & -1
         \end{pmatrix}
         + \beta, \quad \xi(t)=1.
\end{align*}
Then $\xi(x), \xi(y)$ and $1$ generate $A$ as an algebra.

\end{lemma}

\begin{proof}
Observe that $\xi(x^j(x^2-x+1)) = \alpha^j(\alpha^2-\alpha+1)$ (where $\alpha^j(\alpha^2-\alpha+1) = (\alpha_1^j(\alpha_1^2-\alpha_1+1), \dots, \alpha_5^j(\alpha_5^2-\alpha_5+1))$). Now let $W$ be the matrix with entries $W_{ij} = \alpha_i^j(\alpha_i^2-\alpha_i+1)$. Let $V$ be the Vandermonde matrix with entries $V_{ij} = \alpha_i^{j}$. Let $D$ be a diagonal matrix with entries $D_{ij} = \delta_{ij}(\alpha_{i}^2-\alpha_i+1)$. It follows that $W=DV.$ Now since $V$ is nonsingular, $W$ is nonsingular. Now this implies that any $5$-dimensional vector in $A$ is in $\im \xi$. Then $\xi(x)-\alpha\in \im \xi$ and hence
\[
\begin{pmatrix}
           i_1 & 0 \\
           0 & i_2
         \end{pmatrix}  \in \im \xi.
\]
Similarly, $\xi(y)-\beta\in \im \xi$ and hence
\[
\begin{pmatrix}
           1 & e \\
           e & -1
         \end{pmatrix}  \in \im \xi.
\]
Now recall that $i_1,i_2$ are roots of the polynomial $z^3+1$, so we have 
\[
\begin{pmatrix}
           i_1 & 0 \\
           0 & i_2
         \end{pmatrix}^3 =   \begin{pmatrix}
           -1 & 0 \\
           0 & -1
         \end{pmatrix}.
\]
Notice that every diagonal matrix is a linear combination of $\begin{pmatrix}
           i_1 & 0 \\
           0 & i_2
         \end{pmatrix}$ and $\begin{pmatrix}
           -1 & 0 \\
           0 & -1
         \end{pmatrix}$. Now, it can be checked that these matrices are sufficient to generate all matrices in $M_2(\C)$.

\end{proof}

\begin{notation}\label{not:f}
Let $f:\C[x,y][t]\rightarrow A[t]$ be a homomorphism of $\C$-algebras such that
\begin{align*}
    f(x) &= t^2\begin{pmatrix}
           i_1 & 0 \\
           0 & i_2
         \end{pmatrix}
         + t^2 \alpha, \quad
    f(y) = \frac{1}{\sqrt{1 + e^2}} t^3\begin{pmatrix}
           1 & e \\
           e & -1
         \end{pmatrix}
         + t^3 \beta, \quad f(t)=t\1, 
\end{align*}
where $\1$ is the identity element in $A$. We write $\Tilde{1} := (1, 1, 1, 1, 1)$.
As usual we denote by $1$ the identity element in $\C$ and thus in $\C[x, y][t]$. 
\end{notation}

We now prove Theorem \ref{thm:wem}.
\begin{proof}
    Referring to Notation \ref{not:alpha,beta} and \ref{not:f} we calculate:
    \begin{align*}
        f(x^2)&=t^4\Big( \begin{pmatrix}
           i_1^2 & 0 \\
           0 & i_2^2
         \end{pmatrix} +\alpha^2 \Big),\\
        f(y^2)&=t^6\Big( \begin{pmatrix}
           1 & 0 \\
           0 & 1
         \end{pmatrix} +\beta^2 \Big),\\
        f(xy)&=\frac{t^5}{\sqrt{1+e^2}}\Big( \begin{pmatrix}
           i_1 & i_1e \\
           i_2e & -i_2
         \end{pmatrix} +\alpha\beta \Big),\\
        f(yx)&=\frac{t^5}{\sqrt{1+e^2}}\Big( \begin{pmatrix}
           i_1 & i_2e \\
           i_1e & -i_2
         \end{pmatrix} +\alpha\beta \Big).
    \end{align*}

Note that $f(x^2-t^2x+t^4) = t^4(\alpha^2-\alpha+ \Tilde{1})$ and hence
\begin{align*}
    f(x(x^2-t^2x+t^4 )) &= t^6\alpha(\alpha^2-\alpha+\Tilde{1}),\\
    f(x^2(x^2-t^2x+t^4)) &= t^8\alpha^2(\alpha^2-\alpha+\Tilde{1}),\\
    f(y(x^2-t^2x+t^4)) &= t^7\beta(\alpha^2-\alpha+\Tilde{1}).\\
\end{align*}

Observe that
\begin{align*}
    f(xy+yx) &= \frac{t^5}{\sqrt{1+e^2}} \begin{pmatrix}
           2i_1 & e \\
           e & -2i_2
         \end{pmatrix}+ t^5 2 \alpha\beta, \\
    f(t^2y) &= \frac{t^5}{\sqrt{1+e^2}} \begin{pmatrix}
           1 & e \\
           e & -1
         \end{pmatrix}+ t^5 \beta, \\
    f(t^5\frac{i_2-i_1}{\sqrt{1+e^2}}) &= \frac{t^5}{\sqrt{1+e^2}} \begin{pmatrix}
           i_2-i_1 & 0 \\
           0 & i_2-i_1
         \end{pmatrix}+ \frac{t^5(i_2-i_1)}{\sqrt{1+e^2}} \Tilde{1}. \\
\end{align*}

Note that $2i_1-1+(i_2-i_1)=0$ and $-2i_2+1+(i_2-i_1)=0$, so
\[f(xy+yx-t^2y+\frac{t^5(i_2-i_1)}{\sqrt{1+e^2}})= 2\alpha\beta-\beta+ \frac{i_2-i_1}{\sqrt{1+e^2}}\Tilde{1}.
\] 
Now notice that $2\alpha\beta-\beta+ \frac{i_2-i_1}{\sqrt{1+e^2}}\Tilde{1}=0$ by assumptions of Notation \ref{not:alpha,beta}. It follows that $xy+yx-t^2y+\frac{t^5(i_2-i_1)}{\sqrt{1+e^2}} \in I$ and so $xy+yx\in \langle J\rangle$ is a relation in $N$ as defined in Notation \ref{not:N}.

Notice that by Notation \ref{not:alpha,beta} we have
$\beta^2(2\alpha-1)^2=c^2 \Tilde{1}$. Now it follows by Notation \ref{not:alpha} that \begin{equation} \label{alpha_plus_beta_is_zero}
    \alpha^3 + \beta^2 = 0,
\end{equation}
and thus $x^3+y^2\in \langle J\rangle$ is a relation in $N$.

We now proceed to show that $y^3\in \langle J\rangle$. Notice that: 
\[f(y^3) = \frac{t^9}{\sqrt{1+e^2}}\begin{pmatrix}
           1 & e \\
           e & -1
         \end{pmatrix}+t^9\beta^3.\]
Hence $y^3-t^6y = t^9(\beta^3-\beta).$ Therefore it is sufficient to show that $y^3-t^6y+q\in I$ for some $q\in t\C[x,y][t]$. We denote: 
\begin{align*}
    r_1 &:= t^5(x^2-t^2x+t^4) \\ 
    r_2 &:= t^3x(x^2-t^2x+t^4) \\ 
    r_3 &:= tx^2(x^2-t^2x+t^4) \\ 
    r_4 &:= t^2y(x^2-t^2x+t^4)
\end{align*}

Notice that the  elements 
\begin{align*}
    f(r_1) &= t^9(\alpha^2-\alpha+\Tilde{1})=:t^9e_1,\\
    f(r_2) &= t^9\alpha(\alpha^2-\alpha+\Tilde{1})=:t^9e_2,\\
    f(r_3) &= t^9\alpha^2(\alpha^2-\alpha+\Tilde{1})=:t^9e_3,\\
    f(r_4) &= t^9\beta(\alpha^2-\alpha+\Tilde{1})=:t^9e_4
\end{align*}
are in $t\C[x,y][t]$. Note that it suffices to show that $y^3-t^6y+q\in I$ for some $q\in \C r_1+\C r_2+\C r_3+\C r_4$. Therefore it suffices to show that 
\begin{equation}\label{eq:rel1}
    \beta^3-\beta\in  \C e_1+\C e_2+\C e_3+\C e_4 .
\end{equation}
Now relation (\ref{eq:rel1}) is equivalent to 
\[(2\alpha-1)\beta(\beta^2-1)\in  \C (2\alpha-1)e_1+\C (2\alpha-1)e_2+\C (2\alpha-1)e_3+\C (2\alpha-1)e_4,\]
which in turn is equivalent to
\begin{equation}\label{eq:rel2}
    -c(\alpha^3+1) \in  \C (2\alpha-1)e_1+\C (2\alpha-1)e_2+\C (2\alpha-1)e_3+\C (2\alpha-1)e_4.
\end{equation}
Now, notice that $\alpha^3 + 1 = (\alpha+1)(\alpha^2+\alpha+1)$ and that $\alpha+1$ has non-zero entries, so multiplying relation (\ref{eq:rel2}) by $\alpha+1$ we get an equivalent relation
\begin{align*}
    (\alpha^3+1)(\alpha+1)\in  \C &(2\alpha-1)(\alpha^3+1)+\C (2\alpha-1)(\alpha^3+1)\alpha\\
    &+\C (2\alpha-1)(\alpha^3+1)\alpha^2+\C (\alpha^3+1).
\end{align*}
Notice that the right hand side can be rewritten as $\C (\alpha^3+1)+\C (\alpha^3+1)\alpha+\C (\alpha^3+1)\alpha^2+\C (\alpha^3+1)\alpha^3$.


\end{proof}

\section{A more general method}

In this section we consider a more general method that results in algebras $N$ that are not necessarily local and have more generators. Furthermore, we consider $A$ to have finitely many generators.

Let $A$ be a finitely generated $\C$-algebra. 
\begin{enumerate}
    \item We consider a free $\C$-algebra $F$ with $j$ generators $x_1, \dots, x_j$ for some $j \in \N$. We define:
\begin{align*}
    f: F[t] &\to A[t] \\ 
      x_i &\mapsto \sum\limits_{l = 0}^m a_{i, l} t^l
\end{align*}
for some $m \in \N$ and some $a_{i,l}\in A$. We also assume that for each $z\in \mathbb C$, $z\neq 0$ elements  $\sum\limits_{l = 0}^m a_{i, l} z^l$ for $i=1,2,\ldots j$ generate $A$. In other words, images of the generators generate A when $t$ is set to non-zero numbers.

    \item We define \[\mathcal{N} := \frac{\C[x_1, \dots, x_j][t]}{I},\] where $I := \ker(f)$.

    \item We assume that for every $r\in A$ there exists $i=i(r)$ such that $rt^i\in \im f$. 
    \item Our algebra $N = N_0$ is {\bf isomorphic to} 
\[  \frac{\mathcal{N}}{\langle t +I\rangle}\]
 and deforms flatly and associatively to $A$. Moreover $N$ is isomorphic to 
 $\frac{\C[x_1, \dots, x_j][t]}{\langle I, t\rangle},$ where $\langle I, t\rangle $ is the ideal of $\C[x_1, \dots, x_j][t]$ generated by elements of $I$ and by $t$.

 If $j = 2$ then this method corresponds to the method described in section \ref{sec:method}  provided that $a_{i, 0} = 0$ for all $i$. 
\end{enumerate}

{\bf Proof.} 
 When reasoning as in the previous chapters we will use at all places $\mathbb C[x_{1}, \ldots , x_{j}]$ instead of $\mathbb C[x,y]$, and we will use $\mathbb C[x_{1}, \ldots , x_{j}][t]$ instead of $\mathbb C[x,y][t]$.
Let $A$ have dimension $n$ over $\mathbb C$.  Recall that we  \[\mathcal{N} := \frac{\C[x_1, \dots, x_j][t]}{I},\] where $I := \ker(f)$.
\begin{itemize}

\item Reasoning analogously as in Proposition \ref{basis} we obtain that  there are $q_{1}, \ldots , q_{n}\in F$ such that 
\[f(F[t])\subseteq    \sum_{i=1}^{n}{\mathbb C}f(q_{i}),\]
and hence, since $I=\ker f$, we get  
\[F(t)\subseteq  \sum_{i=1}^{n}{\mathbb C}q_{i}+I.\]

In particular, there are $\zeta _{i,k,m}\in \mathbb C$, $\xi_{i,k,m}(t)\in {\mathbb C}[t]$ such that 
 \[q_{k}\cdot q_{m}-(\sum_{i=1}^{n} \zeta _{i,k,m}q_{i}+t\xi_{i,k,m}q_{i})\in I.\]

\item Reasoning similarly as in Theorem \ref{thm formal deformation} 
 we obtain that if $d_{1}, \ldots , d_{n}$ are free generators  then
 \[d_{k}* d_{m}-(\sum_{i=1}^{n} \zeta _{i,k,m}d_{i}+t\xi_{i,k,m}d_{i}),\]
 gives a formal deformation such that $*_{0}$ gives a multiplication on $N$ where
 $N$ is the $\mathbb C$-algebra with generators $d_{1}, \ldots , d_{n}$ subject to relations 
  \[d_{k}*_{0} d_{m}=\sum_{i=1}^{n} \zeta _{i,k,m}d_{i}.\]

Observe that $*$ defines a multiplication on the algebra  $\mathbb C[d_{1}, \ldots , d_{n}][t]/I'$, where $I'$ is the ideal of 
 $\mathbb C[d_{1}, \ldots , d_{n}][t]$ generated by elements 
 \[d_{k}* d_{m}-(\sum_{i=1}^{n} \zeta _{i,k,m}d_{i}+t\xi_{i,k,m}d_{i}).\]
 Notice that $d_{1}, \ldots , d_{n}$ are free generators 
 of the free algebra  $\mathbb C[d_{1}, \ldots , d_{n}]$. 
 Notice that the algebra  $\mathbb C[d_{1}, \ldots , d_{n}][t]/I'$  is isomorphic to algebra 
$\mathcal{N}$. To see this consider the map 
 $\sigma  \mathbb C[d_{1}, \ldots , d_{n}][t]\rightarrow {\mathcal N}$ given by 
\[d_{i}\rightarrow q_{i}+I.\]
 Notice that $I'\subseteq \ker \sigma $ since  
 \[(q_{k}+I)* (q_{m}+I)=\sum_{i=1}^{n} \zeta _{i,k,m}(q_{i}+I)+t\xi_{i,k,m}(q_{i}+I).\]
 Recall  that $\ker \sigma =I'$ because if $a\in ker \sigma $ then 
by using relations from $I'$ we can present $a$ as $a'+i$ where $i'\in I$ and where 
 $a'=\sum_{i=1}^{n}\alpha _{i}(t)d_{i}$ for some $\alpha _{i}(t)\in {\mathbb C}[t]$.
  Notice that then $\sigma (i')=0$ and hence 
\[\sigma (a)=\sigma (a')=\sigma (\sum_{i=1}^{n}\alpha _{i}(t)d_{i})= \sum_{i=1}^{n}\alpha _{i}(t)q_{i}.\]
  Reasoning as in  Corollary $3.3$ we get that  this implies that all $\alpha _{i}=0$, hence $a'=0$, hence $a=i'$. Therefore, $\ker \sigma \subseteq I$ and consequently $\ker \sigma =I'$. 

 Therefore, by the First Isomorphism Theorem for algebras 
 ${\mathbb C}[d_{1}, \ldots , d_{n}][t]/I'$ is isomorphic to $Im (\sigma ) $ which
 is equal to $ {\mathcal N}$. To see that $\sigma $ is surjective, notice that
 $q_{1}+I, \ldots , q_{n}+I$ span $\mathcal N$ as a linear space 
(as mentioned before this can be proved similarly as Proposition \ref{basis}). 

\item  Reasoning as in Proposition \ref{3} we obtain that $N$ is $n$ dimensional $\mathbb C$-algebra. 
\item Reasoning similarly as in Proposition \ref{7}  
 we obtain that the algebra $N$ is isomorphic to the algebra  $\frac{\C[x_1, \dots, x_j][t]}{\langle I, t\rangle }$ where ${\langle I, t\rangle }$ is the ideal of  $\C[x_1, \dots, x_j][t]$ generated by elements from $I$ and by $t$. By the Third Isomorphism Theorem
 $N$ is isomorphic to 
 \[  \frac{\mathcal{N}}{\langle t +I\rangle}.\]
 This can be seen by reasoning similarly as in Proposition \ref{prop:10} but taking $t$ instead of $t-1$ in the proof. 

\item We will now consider deformations $*_{1}$ at $t=1$. Consider first the homomorphism 
 of  $\mathbb C$-algebras $\xi : F[t]\rightarrow A$ given by  
$\xi(x_{i})=\sum_{l=0}^{m}a_{i,l}$ for $i=1,2, \ldots , j$.

 Reasoning similarly as in Proposition \ref{iso} we obtain that $\ker \xi =\langle I, t-1\rangle$. 
 Reasoning analogously as in the last two lines of Proposition \ref{prop:10} we obtain that 
$F[t]/ \ker \xi $ is isomorphic to $A$.

\item Reasoning as in Theorem \ref{thm:iso1} we see that the algebra 
${\mathbb C}[d_{1}, \ldots , d_{n}][t]/I$ deforms at $t=1$ to
the algebra $F[t]/{\langle I, t-1\rangle} =F[t]/{\ker \xi }$ which is isomorphic to $A$.

\item We will now consider deformations $*_{z}$ at $t=z$ where $z\in \mathbb C, z\neq 0,$  $ z\neq 1$.    
Consider first the homomorphism 
 of  $\mathbb C$-algebras $\xi  : F[t]\rightarrow A$ given by  
$\xi (x_{i})=\sum_{l=0}^{m}a_{i,l}z ^{l}$ for $i=1,2, \ldots , j$. 

 Reasoning similarly as in Proposition \ref{iso} we obtain that $\ker \xi =\langle I, t-z\rangle $ (this can be done by using $t'=t/z$ at the place of $t$ in the proof of Proposition \ref{iso}, and in the first line declaring that $p_{i}\in F$ instead of $p_{i}\in {\mathbb C}[x,y]$). 
 
 Next, reasoning analogously as in the last two lines of Proposition \ref{prop:10} we obtain that 
$F[t]/ \ker \xi $ is isomorphic to $A$ (this follows since elements $\xi (x_{i})$ generate $A$).

 Next, reasoning as in Theorem \ref{thm:iso1} we see that the algebra 
${\mathbb C}[d_{1}, \ldots , d_{n}][t]/I$ deforms at $t=z$ to
the algebra $F[t]/{\langle I, t-z\rangle} =F[t]/{\ker \xi }$ which is  is isomorphic to $A$. \qed

\end{itemize}

\section{Future Work}
In this section we pose some questions about Method \ref{meth} and suggest some generalisations.

\begin{qu}
    Let $A$ be a $\C$-algebra and suppose $G_1 = \{a, b\}$ and $G_2 = \{a', b'\}$ are two distinct generating sets of $A$. Under what conditions does Method \ref{meth} using $A, G_1$ and $A, G_2$ result in the same algebra $N_0$ that has $A$ as a deformation?  
\end{qu}

\begin{qu}
    Let $A$ be a $\C$-algebra. Under what conditions does Method \ref{meth} only produce one algebra $N$ that has $A$ as a deformation. 
\end{qu}

\begin{qu}
    Let $A$ be a $\C$-algebra. Does their exist a finitely terminating algorithm that will produce all the algebras $N$ arising from Method \ref{meth} that have $A$ as a deformation? 
\end{qu}

\begin{qu}
    Let $A$ be a $\C$-algebra. Do ring theoretic properties of the generators $a, b$ (for example being idempotent, being irreducible idempotent, prime) determine any properties of the algebra $N$ arising from Method \ref{meth} using $A, \{a, b\}$?
\end{qu}

\begin{qu}
    Let $A$ be a $\C$-algebra. Do ring theoretic properties of $A$ (for example being semi-simple, local) determine any properties of the algebra $N$ arising from Method \ref{meth} using $A$?
\end{qu}

\begin{qu}
    Can Method \ref{meth} be generalised to be used for algebras over a general commutative ring?
\end{qu}

\begin{center}
\subsubsection*{\sc{Acknowledgements}} 
\end{center}
We are grateful to Michael Wemyss  for suggesting his interesting questions which inspired this paper, and for useful comments about contraction algebras and their applications in geometry.
The third author acknowledges support from the EPSRC programme grant EP/R034826/1 and from the EPSRC research grant EP/V008129/1.

\cleardoublepage
\phantomsection

\addcontentsline{toc}{chapter}{Bibliography}
\bibliographystyle{alpha}
\bibliography{bibliography}

\end{document}